\documentclass[10pt]{article}
\usepackage{amsmath, amsthm}\usepackage{enumerate}
\usepackage{amssymb}\usepackage{bbold}
\usepackage{color}
\newtheorem{theorem}{Theorem}%[section]
\newtheorem{definition}{Definition}%[section] 
\newtheorem{exam}{Example}%[section]
\newtheorem{prop}{Proposition}%[section]
\newtheorem{lem}{Lemma}%[section]
\newtheorem{cor}{Corollary}%[section]
\DeclareMathOperator{\convo}{\xrightarrow[]{o}}

\DeclareMathOperator{\convn}{\xrightarrow[]{\|\cdot\|}}

\makeatletter
\renewcommand{\subsection}{\@startsection{subsection}{1}
{0pt}{3.25ex plus 1ex minus.2ex}{-1em}{\normalfont\normalsize\bf}}
\makeatother
\begin{document}

\title{{\bf Algebras of Lebesgue and KB regular operators in Banach lattices}}
\date{}
\maketitle
\author{\centering{{Eduard Emelyanov$^{1}$\\ 
\small $1$ Sobolev Institute of Mathematics, Novosibirsk, Russia}

\abstract{It is shown that the regularly Lebesgue and regularly (quasi) $KB$ operators 
in a Banach lattice $E$ form operator algebras and, for Dedekind complete $E$, 
even Banach lattice algebras. We also prove that a Banach lattice $E$ with order continuous 
norm is a $KB$-space iff each positive compact operator in $E$ is a $KB$ operator.}\\


\large

\bigskip

%%%%%%%%%%%%%%%%%%%%
\section{Introduction}
%%%%%%%%%%%%%%%%%%%%

Several operator versions of Banach lattice properties 
like the property to be a $KB$-space were introduced and 
studied in recent papers \cite{AlEG,BA,JAM}. 
The main idea of such operator versions lies in 
a redistribution of topological and order properties 
between domain and range of operators. Since 
the order convergence is not topological in infinite 
dimensional vector lattices \cite{DEM}, the important 
operator versions emerge when order and norm convergences 
are involved simultaneously. The case of regular operators 
acting in a Banach lattice is of special interest. 

\subsection{} 
In the present paper, $E$ is a real Banach lattice and $\text{\rm L}(E)$
is the algebra of all continuous linear operators in $E$.
The space $\text{\rm L}_r(E)$ of regular 
operators in $E$ is an ordered subalgebra of $\text{\rm L}(E)$. 
In general, $\text{\rm L}_r(E)$ is not a vector lattice. 
However, $\text{\rm L}_r(E)$ is a Banach space 
under the {\em regular norm} 
$\|T\|_r=\inf\{\|S\|:\pm T\le S\in\text{\rm L}_r(E)\}$ \cite[Proposition 1.3.6]{Mey}.
%Furthermore, for every $x\in E$,
%$\|T\|_r=\inf\{\|S\|: S\in\text{\rm L}(E), |Tx|\le S|x|\ \forall x\in E\}$
%whenever $T\in{\cal L}_r(E)$. 
If $E$ is Dedekind complete, then $(\text{\rm L}_r(E),\|\cdot\|_r)$
is a Banach lattice algebra, and 
$$
   \|T\|_r=\|~|T|~\| \ \ \ \ \ (\forall T\in\text{\rm L}_r(E)).
   \eqno(1)
$$

\subsection{}
The following is an adopted variant of Definition 1.1 
from \cite{AlEG} in the Banach lattice setting.

\begin{definition}\label{order-to-topology} {\em
An operator $T\in\text{\rm L}(E)$ is said to be:
\begin{enumerate}[$(i)$]
\item \
{\em Lebesgue} ({\em $\sigma$-Lebesgue}) if $\|Tx_\alpha\|\to 0$ 
for every net (sequence) $(x_\alpha)_\alpha$ in $E$ such that 
$x_\alpha\downarrow 0$.
The set of all such operators is denoted by $\text{\rm L}_{Leb}(E)$ 
(resp., $\text{\rm L}^\sigma_{Leb}(E)$).
\item \  
{\em $on$-continuous} ({\em $on\sigma$-continuous}) if 
$\|Tx_\alpha\|\to 0$ for every net (sequence) $(x_\alpha)_\alpha$ in $E$ 
such that $x_\alpha\convo 0$ (cf., \cite[p.1315]{JAM}).
The set of all such operators is denoted by $\text{\rm L}_{on}(E)$ 
(resp., $\text{\rm L}^\sigma_{on}(E)$).
\item \  
{\em $KB$} ({\em $\sigma$-$KB$}) if, for every norm-bounded 
increasing net (sequence) $(x_\alpha)_\alpha$ in $E_+$,
there exists (not necessarily unique) $x\in E$ such that $\|Tx_\alpha-Tx\|\to 0$.
The set of all such operators is denoted by $\text{\rm L}_{KB}(E)$ 
(resp., $\text{\rm L}^\sigma_{KB}(E)$).
\item \  
{\em quasi $KB$} ({\em quasi $\sigma$-$KB$}) if $T$ takes 
norm-bounded increasing nets (sequences) in $E_+$ to norm Cauchy nets.
The set of all such operators is denoted by $\text{\rm L}_{qKB}(E)$ 
(resp., $\text{\rm L}^\sigma_{qKB}(E)$) (cf. \cite[Definition 1.1]{BA}).
\end{enumerate}
}
\end{definition}
\noindent
The identity operator in $c_0$ is Lebesgue yet not $KB$.
The next example shows that $KB$ operators also need 
not to be Lebesgue.

\begin{exam}\label{Exam all bounded real functions on [0,1]}%%%OK
{\em
Let $E$ be a Banach lattice of all bounded real functions on $[0,1]$ 
that differ from a constant on at most countable subset of $[0,1]$. 
Let $T:E \to E$ be an operator that assigns to each $f\in E$ the constant function 
$Tf$ on $[0,1]$ such that the set $\{x\in [0,1] : f(x) \ne (Tf)(x)\}$ is at most 
countable. Since $T$ is rank one and continuous, $T$ is $KB$. 
Consider the net $(f_\alpha)_\alpha$ in $E$, indexed by 
finite subsets of $[0,1]$, defined by: 
$ 
  f_\alpha(x) = \left\{
  \begin{array}{ccc}
   1 &\text{ if } & x \not\in \alpha\\
   0 &\text{ if } & x \in \alpha
  \end{array}
\right..
$
Then $f_\alpha \downarrow 0$ in $E$, yet $\|f_\alpha \|_\infty=1$ 
for all $\alpha$. Thus $T$ is not a Lebesgue operator. 
However, $T$ is $\sigma$-Lebesgue.
}
\end{exam}
\noindent
The following example shows that quasi $KB$ operators need not to be $KB$.

\begin{exam}\label{quasi KB yet not KB}%%%OK
{\em
Let $T$ be an operator in $E=C[0,1]\oplus L_1[0,1]$ defined by
$T\bigl((\phi,\psi)\bigl)=(0,\phi)$ for $\phi\in C[0,1]$ and $\psi\in L_1[0,1]$. 
Clearly, $T$ is 
a quasi $KB$ operator. But $T$ is not $KB$.
Indeed, let a continuous function $\phi_n$ be equal 
to 1 on $[0,\frac{1}{2}-\frac{1}{2^n}]$,
to 0 on $[\frac{1}{2},1]$, and linear otherwise.
Let $f_n:=(\phi_n,0)$. Then $f_n\uparrow$ in $E$ with  
$\|f_n\|\le 1$ for all $n\in\mathbb{N}$ and $Tf_n\convn(0,g)$,
where $g\in L_1[0,1]$ is the indicator function of $[0,\frac{1}{2}]$.
Since $g\not\in C[0,1]$, there is no element $x$ of $E$ satisfying
$Tx=(0,g)$, and hence $T$ is not $KB$.
}
\end{exam}

\subsection{}
The operators belonging to the linear span of positive compact operators
({\rm r-}compact operators) were investigated by Z.L.~Chen and A.W.~Wickstead in \cite{CW97}.
The class of {\rm r-}compact operators is closed under neither the operator nor
regular norm. However, there is a special norm on {\rm r-}compact operators
that makes this class closed (see \cite[Proposition 2.2]{CW97}). This norm occurs 
in application of Proposition 1.3.6 of \cite{Mey} to compact rather 
than continuous operators. The general case, that is a quite straightforward 
extension of \cite[Proposition 1.3.6]{Mey}, is described in
Lemma \ref{P-norm} below. First, we need the following notion.

\begin{definition}\label{rP-operators} {\em
Given a set {\rm P} of operators in $E$
$($briefly, the set {\rm P} consists of {\rm P}-operators$)$. An operator
$T:E\to E$ is called a {\em regularly {\rm P}-operator} $($shortly, an
{\em {\rm r-P}-operator}$)$, whenever there exist two positive  
{\rm P}-operators $T_1,T_2$ such that $T=T_1-T_2$.
}
\end{definition}
\noindent
As examples of regularly P-operators one can consider: 
r-{\em positive} and r-{\em regular operators} 
that are just regular; r-{\em finite rank operators}; 
r-{\em $\sigma$-Lebesgue};
r-{\em compact}; r-$KB$ {\em operators}; etc. 

The importance of regularly P-operators in the present paper is justified by the fact that
most of the key results below, like Propositions \ref{useful}, \ref{prop2}, yields for
\text{\rm r-P}-operators.

\begin{lem}\label{P-norm}
Let ${\rm P}$ be a closed in the operator norm subspace 
of $\text{\rm L}(E)$. Then the formula 
$$
   \|T\|_\text{\rm r-P}=\inf\{\|S\|:\pm T\le S\in{\rm P}\}
   \eqno(2)
$$
defines a norm on the space $\text{\rm r-P}(E)$ of all
\text{\rm r-P}-operators in $E$. Furthermore,
$\|T\|_{\text{\rm r-P}}\ge\|T\|_r\ge\|T\|$ for all $T\in\text{\rm r-P}(E)$, and  
$(\text{\rm r-P}(E),\|\cdot\|_{\text{\rm r-P}})$ \ is a Banach space.
\end{lem}

\begin{proof}
We skip routine checking that $\|\cdot\|_{\text{\rm r-P}}$ is a norm on 
$\text{\rm r-P}(E)$ satisfying $\|\cdot\|_{\text{\rm r-P}}\ge\|\cdot\|_r\ge\|\cdot\|$. 
Take any $\|\cdot\|_{\text{\rm r-P}}$-Cauchy sequence $(T_n)_n$ in $\text{\rm r-P}(E)$.
Then $T_n=G_n-R_n$ for some positive {\rm P}-operators $G_n,R_n$.
We can assume also that $\|T_{n+1}-T_n\|_{\text{\rm r-P}}<2^{-n}$ for all
$n\in\mathbb{N}$. As $\|\cdot\|_{\text{\rm r-P}}\ge\|\cdot\|$, there exists
$T\in\text{\rm L}(E)$ such that $\|T-T_n\|\to 0$.
Since $\text{\rm P}$ is closed in the operator norm, $T\in\text{\rm P}$.
Pick $S_n\in\text{\rm P}$ with $\|S_n\|<2^{-n}$ and $\pm(T_{n+1}-T_n)\le S_n$. Then 
$$
   T_{n+1}x^+-T_nx^+\le Sx^+ \ \ \text{and} \ \ 
   -T_{n+1}x^-+T_nx^-\le Sx^- 
   \eqno(3) 
$$
for each $x\in E$. Summing up the inequalities in (3) gives 
$T_{n+1}x-T_nx\le S_n|x|$. Replacing $x$ by $-x$ 
gives $T_nx-T_{n+1}x\le S_n|x|$, and hence
$$
   |(T_{n+1}-T_n)x|\le S_n|x| \ \ \ (\forall x\in E).
   \eqno(4) 
$$
As $\text{\rm P}$ is closed in the operator norm, 
$Q_n:=\sum\limits_{k=n}^\infty S_k\in\text{\rm P}$ for all $n$.
By (4), we obtain 
$$
   |(T-T_n)x|=\lim\limits_{k\to\infty}|(T_k-T_n)x|\le
$$
$$
   \sum\limits_{k=n}^\infty|(T_{k+1}-T_n)x|\le Q_n|x| \ \ \ \ \ (x\in E), 
$$
and hence $\pm(T-T_n)\le Q_n$. Then
$$
  -Q_n\le(T-T_n)\le Q_n\ \text{\rm and} \ 0\le(T-T_n)+Q_n
$$
for all $n\in\mathbb{N}$. Therefore 
$$
   T=[(T-T_n)+Q_n]+[T_n-Q_n]=
$$
$$
   [(T-T_n)+Q_n+G_n]-[R_n+Q_n]\in\text{\rm r-P}(E),
$$
and hence $(T-T_n)\in\text{\rm r-P}(E)$ for all $n\in\mathbb{N}$.
We conclude from $\|T-T_n\|_{\text{\rm r-P}}\le\|Q_n\|<2^{1-n}$ 
that $T_n\stackrel{\|\cdot\|_{\text{\rm r-P}}}{\longrightarrow}T$.
\end{proof}

%%%%%%%%%%%%%%%%%%%%
\section{Algebraic and lattice properties of regularly Lebesgue and KB operators}
%%%%%%%%%%%%%%%%%%%%

\subsection{}
The following proposition contains a  useful simple observation 
(cf. \cite[Lemma 1]{JAM} and \cite[Lemma 2.1]{AlEG}).

\begin{prop}\label{useful}%%%OK
Every \text{\rm r}-Lebesgue $(resp. \ $\text{\rm r}-$\sigma$-Lebesgue$)$
operator is order continuous $($resp. $\sigma$-con\-ti\-nuous$)$.
\end{prop}

\begin{proof}
Let $T$ be a regularly Lebesgue operator in $E$,
and let $x_\alpha\downarrow 0$ in $E$. As $T=T_1-T_2$ with
both $T_1$ and $T_2$ positive Lebesgue operators,  
$T_1x_\alpha\downarrow$ and $\|T_1x_\alpha\|\to 0$. 
By \cite[Theorem 2.14]{AB1}, $T_1x_\alpha\downarrow 0$. 
Also, $T_2x_\alpha\downarrow 0$. Then $Tx_\alpha\convo 0$.
The \text{\rm r}-$\sigma$-Lebesgue case is similar.
\end{proof}

\begin{lem}\label{L1}%%%OK
Let a net $(T_\gamma)_\gamma$ in $\text{\rm L}_{Leb}(E)$ 
satisfy $T_\gamma\stackrel{\|\cdot\|}{\longrightarrow}T$
for some $T\in\text{\rm L}(E)$. Then $T\in\text{\rm L}_{Leb}(E)$.
\end{lem}

\begin{proof}
Let $x_\alpha\downarrow 0$ in $E$. 
Passing to a tail of the net $(x_\alpha)_\alpha$, 
we may assume that $\|x_\alpha\|\le M$ for some $M>0$
and for all $\alpha$. Take an $\varepsilon>0$. 
Let $\gamma_\varepsilon$ be such that 
$\|T_{\gamma_\varepsilon}-T\|\le\frac{\varepsilon}{M}$. 
Since $T_{\gamma_\varepsilon}$ is Lebesgue, 
$\lim_\alpha\|T_{\gamma_\varepsilon}x_\alpha\|=0$. 
There exists $\alpha_\varepsilon$ with 
$\|T_{\gamma_\varepsilon}x_\alpha\|\le\varepsilon$ 
for $\alpha\ge\alpha_\varepsilon$, and hence
$$
   \|Tx_\alpha\|\le\|T_{\gamma_\varepsilon}x_\alpha\|+ 
   \|T_{\gamma_\varepsilon}x_\alpha-Tx_\alpha\|\le
   \varepsilon+\|T_{\gamma_\varepsilon}-T\| 
   \|x_\alpha\|\le 2\varepsilon
$$
for $\alpha\ge\alpha_\varepsilon$. Since $\varepsilon>0$ is arbitrary, 
$T$ is Lebesgue. 
\end{proof}
\noindent
Recall that P-operators satisfy the {\em domination property} if 
$0\le S\le T\in\text{\rm P}$ implies $S\in\text{\rm P}$.

\begin{theorem}\label{th1}%%%OK
Let $E$ be a Banach lattice. Then $\text{\rm r-L}_{Leb}(E)$ 
and $\text{\rm r-L}_{Leb}^\sigma(E)$ are both left algebraic ideals 
and hence subalgebras of $\text{\rm L}_r(E)$. Furthermore$:$ 
\begin{enumerate}[$i)$]
\item \
$\text{\rm r-L}_{Leb}(E)$
$($resp., $\text{\rm r-L}_{Leb}^\sigma(E)$$)$
coincides with $\text{\rm L}_r(E)$ iff the norm of $E$ is order 
continuous $($resp., $\sigma$-order continuous$)$. 
\item \
If $E$ is Dedekind complete then $\|T\|_{\text{\rm r-P}}=\|T\|_r$
for each $T\in P:=\text{\rm L}_{Leb}(E)$, and 
$$
  (\text{\rm r-L}_{Leb}(E), \ \|\cdot\|_{\text{\rm r-P}}) \ \ \ \ 
  (\text{\rm resp.,} \ 
  (\text{\rm r-L}^\sigma_{Leb}(E), \ \|\cdot\|_{\text{\rm r-P}}))
$$
is a Banach lattice algebra.
\end{enumerate}
\end{theorem}

\begin{proof}
We consider only $\text{\rm r-L}_{Leb}(E)$, 
the case of $\text{\rm r-L}^\sigma_{Leb}(E)$ is similar,
and skip routine direct checking that $\text{\rm r-L}_{Leb}(E)$
is a vector subspace of $\text{\rm L}_r(E)$. 
Let $T\in\text{\rm r-L}_{Leb}(E)$ and $S\in\text{\rm L}_r(E)$. 
In order to show that $\text{\rm r-L}_{Leb}(E)$ is a left 
algebraic ideal, we need to show that $ST$ is r-Lebesgue. 
WLOG, we may assume $S,T\ge 0$.
Let $x_\alpha\downarrow 0$ in $E$. Then $\|Tx_\alpha\|\to 0$, 
$\|S(Tx_\alpha)\|\to 0$, and hence $ST$ is r-Lebesgue. 
As each left ideal is a subalgebra, $\text{\rm r-L}_{Leb}(E)$ 
is a subalgebra of $\text{\rm L}_r(E)$. Note that by the similar argument,
$\text{\rm r-L}_{Leb}(E)$ is a left ideal of $\text{\rm L}(E)$.

$i)$ \
If $\text{\rm r-L}_{Leb}(E)=\text{\rm L}_r(E)$ then the identity 
operator $I$ in $E$ is Lebesgue, and hence the norm in $E$ is order continuous. 
Conversely, if the norm in $E$ is order continuous then 
$I\in\text{\rm r-L}_{Leb}(E)$ and, 
since $\text{\rm r-L}_{Leb}(E)$ is a left ideal of $\text{\rm L}_r(E)$
containing $I$ then $\text{\rm r-L}_{Leb}(E)=\text{\rm L}_r(E)$.

$ii)$ \
By Lemmas \ref{P-norm} and \ref{L1}, 
$(\text{\rm r-L}_{Leb}(E), \ \|\cdot\|_{\text{\rm r-P}})$ is a Banach space.
Let $T\in\text{\rm r-L}_{Leb}(E)$, say $T=T_1-T_2$,
with $0\le T_1,T_2\in\text{\rm L}_{Leb}(E)$. Then $|T|\le T_1+T_2$.
Since the Lebesgue operators satisfy the domination property,
$|T|\in\text{\rm r-L}_{Leb}(E)$. Hence $\text{\rm r-L}_{Leb}(E)$ is a vector lattice,
and $\|T\|_{\text{\rm r-P}}=\|~|T|~\|$.
By formula (1), $\|T\|_{\text{\rm r-P}}=\|T\|_r$. 

For proving that $(\text{\rm r-L}_{Leb}(E), \ \|\cdot\|_{\text{\rm r-P}})$
is a Banach lattice, we need to show 
$$
   |S|\le|T| \ \Rightarrow \ \|S\|_{\text{\rm r-P}}\le\|T\|_{\text{\rm r-P}}
   \ \ \ (\forall S,T\in\text{\rm r-L}_{Leb}(E)).
   \eqno(5)
$$
Let $S,T\in\text{\rm r-L}_{Leb}(E)$ satisfy $|S|\le|T|$.
Since $|S|,|T|\in{\text{\rm r-}}{\text{\cal L}}_{Leb}(E)$ then
$\|S\|_{\text{\rm r-P}}=\|~|S|~\|\le\|~|T|~\|=\|T\|_{\text{\rm r-P}}$,
that proves (5).

It remains to show that the norm $\|\cdot\|_{\text{\rm r-P}}$ is submultiplicative.
Let $S,T\in\text{\rm r-L}_{Leb}(E)$. It follows from
$$
   \pm STx\le|STx|\le|S||Tx|\le|S||T|x \ \ \ (\forall x\in E_+)
$$
that $\pm ST\le |S||T|$. Thus, 
$$
   \|ST\|_{\text{\rm r-P}}\le\|~|S||T|~\|\le
   \|~|S|~\|\cdot\|~|T|~\|=\|S\|_{\text{\rm r-P}}\|T\|_{\text{\rm r-P}},
$$
as desired.
\end{proof}
%\noindent
%The question whether or not the norm $\|\cdot\|_{\text{\rm r-P}}$ is submultiplicative
%for $\text{\rm P}=\text{\rm L}_{Leb}(E)$ in arbitrary Banach lattice $E$ remains open.

\subsection{}
The following fact is an adaptation of \cite[Lemma 1]{JAM} and \cite[Lemma 2.1]{AlEG} 
for regular operators. For convenience, we include its proof.

\begin{prop}\label{cor1}%%%OK
Let $E$ be a Banach lattice. Then 
$$
   \text{\rm r-L}_{Leb}(E)=\text{\rm r-L}_{on}(E) \ \ \text{and} \ \
   \text{\rm r-L}^\sigma_{Leb}(E)=\text{\rm r-L}^\sigma_{on}(E),
   \eqno(6)
$$
where $\text{\rm r-L}_{on}(E)$ 
$($resp., $\text{\rm r-L}^\sigma_{on}(E)$$)$ is the space of all
\text{\rm r-on}-continuous $($resp., \text{\rm r-on}$\sigma$-continuous$)$ operators in $E$.
\end{prop}

\begin{proof}
We restrict ourselves to the first equality. Since 
$\text{\rm r-L}_{on}(E)\subseteq\text{\rm r-L}_{Leb}(E)$
is trivial, we only need to show  
$\text{\rm r-L}_{Leb}(E)\subseteq\text{\rm r-L}_{on}(E)$.
Let $T\in\text{\rm r-L}_{Leb}(E)$, and let $x_\alpha\convo 0$. 
WLOG, suppose $T\ge 0$. 
There exists a net $y_\beta\downarrow 0$
in $E$ such that, for every $\beta$, there exists $\alpha_\beta$ with
$|x_\alpha|\le y_\beta$ (hence $T|x_\alpha|\le Ty_\beta$)
for $\alpha\ge\alpha_\beta$.
It follows $0 \le T((x_\alpha)^+),T((x_\alpha)^-) \le Ty_\beta$ for 
$\alpha \ge \alpha_\beta$. Then
$\|Ty_\beta\|\to 0$ and hence $\|T((x_\alpha)^\pm)\|\to 0$, 
that implies $\|Tx_\alpha\|\to 0$.
\end{proof}
\noindent
In view of Proposition~\ref{cor1}, one can replace in 
Theorem~\ref{th1} regularly Lebesgue $($$\sigma$-Lebesgue$)$ 
operators by regularly $on$-continuous ($on\sigma$-continuous) operators.

The author doesn't know whether or not one can replace regularly Lebesgue operators
by Lebesgue regular operators in formula (6). More precisely, under what 
conditions on $E$, the equalities
$$
   (\text{\rm L}_{Leb})_r(E)=(\text{\rm L}_{on})_r(E) \ \ \text{\rm and} \ \ 
   (\text{\rm L}^\sigma_{Leb})_r(E)=(\text{\rm L}^\sigma_{on})_r(E)
$$
hold true?  By Proposition~\ref{cor1}, the question has a positive 
answer for a Dedekind complete Banach lattice $E$, where 
$$
   T\in(\text{\rm L}_{Leb})_r(E)\Longrightarrow|T|\in\text{\rm L}_{Leb}(E)
$$
$$
   (\text{\rm resp.,} \ 
   T\in(\text{\rm L}^\sigma_{Leb})_r(E)\Longrightarrow|T|\in\text{\rm L}^\sigma_{Leb}(E)\ ).
$$

\subsection{}
Now we investigate algebraic structure of 
$KB$ and quasi $KB$ operators. We begin with the 
following lemma, that is similar to Lemma \ref{L1} 
(cf. also \cite[Proposition 2.5]{BA}). 

\begin{lem}\label{L2}%%%OK
Let a net $(T_\gamma)_\gamma$ in $\text{\rm L}_{qKB}(E)$ satisfies
$T_\gamma\stackrel{\|\cdot\|}{\longrightarrow}T$
for some $T\in\text{\rm L}(E)$. Then $T\in\text{\rm L}_{qKB}(E)$.
\end{lem}

\begin{proof}
Let $(x_\alpha)_\alpha$ be an increasing norm bounded net in $E$. 
WLOG, $\|x_\alpha\|\le 1$ for all $\alpha$. 
Take $\varepsilon>0$ and choose $\gamma_0$ satisfying  
$\|T_{\gamma_0}-T\|\le\varepsilon$.
Choose $\alpha_0$ such that 
$\|T_{\gamma_0}x_{\alpha^{'}}-T_{\gamma_0}x_{\alpha^{''}}\|\le\varepsilon$
for $\alpha^{'}, \alpha^{''}\ge\alpha_0$. Then 
$$
   \|Tx_{\alpha^{'}} - Tx_{\alpha^{''}} \| \le 
$$
$$
   \|Tx_{\alpha^{'}} - T_{\gamma_0} x_{\alpha^{'}} \|+ 
   \|T_{\gamma_0} x_{\alpha^{'}} - T_{\gamma_0} x_{\alpha^{''}}\|+ 
   \|T_{\gamma_0} x_{\alpha^{''}} - Tx_{\alpha^{''}} \| \le
$$
$$ 
   \|T_{\gamma_0} - T\| \|x_{\alpha^{'}}\| +  
   \|T_{\gamma_0} x_{\alpha^{'}} - T_{\gamma_0} x_{\alpha^{''}} \|+ 
   \|T_{\gamma_0} - T\| \|x_{\alpha^{''}}\|    \le 3\varepsilon
$$
for all $\alpha^{'}, \alpha^{''}\ge\alpha_0$, and hence $T$ is quasi $KB$. 
\end{proof}
\noindent
By \cite[Proposition 1.2]{AlEG}, every quasi $\sigma$-$KB$ operator in $E$
is quasi $KB$. The following theorem is a quasi KB version 
of Theorem \ref{th1} (cf. also \cite[Corollary 2.8]{BA}).

\begin{theorem}\label{th2}%%%OK
$\text{\rm r-L}_{KB}(E)$ 
$($resp., $\text{\rm r-L}_{qKB}(E)$$)$ 
is a left $($resp., two sides$)$ algebraic ideal of  
$\text{\rm L}_r(E)$, and hence is a subalgebra of $\text{\rm L}_r(E)$. 
Furthermore$:$ 
\begin{enumerate}[$i)$]
\item \
The algebra $\text{\rm r-L}_{KB}(E)$ is unital iff 
$\text{\rm r-L}_{qKB}(E)$ is unital iff $E$ is a KB-space.
\item \
If $E$ is Dedekind complete then
$(\text{\rm r-L}_{qKB}(E), \ \|\cdot\|_{\text{\rm r-P}})$
is a Banach lattice algebra for $\text{\rm P}=\text{\rm L}_{qKB}(E)$.
\end{enumerate}
\end{theorem}

\begin{proof}
We omit the straightforward proof that  
$\text{\rm r-L}_{KB}(E)$ and 
$\text{\rm r-L}_{qKB}(E)$
are both vector subspaces of $\text{\rm L}_r(E)$.
We need to show that  $\text{\rm r-L}_{KB}(E)$
(resp., $\text{\rm r-L}_{qKB}(E)$) is closed under left (two sides) 
multiplication by elements of $\text{\rm L}_r(E)$. 

A) 
Firstly consider $\text{\rm r-L}_{KB}(E)$.
Let $T\in\text{\rm r-L}_{KB}(E)$ and $S\in\text{\rm L}_r(E)$. 
WLOG, suppose $T\ge 0$ and $S\ge 0$. If $(x_\alpha)_\alpha$
be an increasing norm bounded net in $E_+$, then $\|Tx_\alpha-Tx\| \to 0$
for some $x\in E$ and, since $S$ is continuous, we obtain 
$\|STx_\alpha-STx\|\to 0$. Then $ST\in\text{\rm r-L}_{KB}(E)$, 
showing that $\text{\rm r-L}_{KB}(E)$ is a left ideal of $\text{\rm L}_r(E)$.

B) 
Now, let $T\in\text{\rm r-L}_{qKB}(E)$ and $S\in\text{\rm L}_r(E)$,
and let $(x_\alpha)_\alpha$ be a norm bounded increasing net in $E_+$. 
Then the net $(Tx_\alpha)_\alpha$ is $\|.\|$-Cauchy and, since $S$ is continuous, 
the net $(STx_\alpha)_\alpha$ is also $\|.\|$-Cauchy, showing 
$ST\in\text{\rm r-L}_{qKB}(E)$. 

Since the net $Sx_\alpha$ is increasing and norm bounded in $E_+$, 
we obtain that the net $(TSx_\alpha)_\alpha$ is $\|.\|$-Cauchy, showing that 
$TS\in\text{\rm r-L}_{qKB}(E)$.

Thus, $\text{\rm r-L}_{qKB}(E)$
is a two sides algebraic ideal of $\text{\rm L}_r(E)$. 

\vspace{3mm}
$i)$
It is obvious since the three conditions are equivalent to
the fact that the identity operator in $E$ is KB.
 
\vspace{3mm}
$ii)$ \
Upon the fact that the KB operators satisfy 
the domination property \cite[Theorem 2.6]{AlEG},
the proof is similar to the proof of $ii)$ of Theorem~\ref{th1}
via replacement of Lemma~\ref{L1} by Lemma~\ref{L2}.
\end{proof}

%%%%%%%%%%%%%%%%%%%%
\section{Relations between KB and compact operators}
%%%%%%%%%%%%%%%%%%%%

Here we discuss relations between KB and compact operators in $E$.

\subsection{}
We begin with the following fact.

\begin{prop}\label{prop1}%%%OK
Every continuous finite rank operator 
on a  Banach lattice $E$ is a regularly $KB$ operator. 
\end{prop}

\begin{proof}
Let $T\in\text{\rm L}(E)$ and $\dim(TE)<\infty$. Then 
$$
   T = \sum_{k=1}^n x_k \otimes f_k 
   \  \text{for} \   x_1, \dots, x_n \in E
   \ \text{and} \   f_1, \dots, f_n \in E^\ast.
$$
WLOG, we can assume $T=x_1\otimes f_1$.
Since the dual $E^\ast$ is Dedekind complete, 
$f_1$ is regular, and hence $T$ is also regular. 
WLOG, suppose  $x_1\ge 0$ 
and $f_1 \ge 0$. Then $T$ is positive, and hence regular.
Let $(x_\alpha)_\alpha$ be an increasing
norm bounded net in $E_+$. Since 
$Tx_\alpha=(x_1 \otimes f_1) (x_\alpha)=f_1 (x_\alpha)x_1\uparrow$ 
and norm bounded, then $Tx_\alpha\to y\in\text{\rm span}(x_1)$ 
and $Tx_\alpha\convn Tx$ for some $x\in E$. 
Hence $T$ is \text{\rm r-KB}.
\end{proof}

It is worth to notice that, for a positive compact operator 
$T$ in the Banach lattice $c$ of real convergent sequences, 
and, for a norm bounded increasing net $(x_\alpha)_\alpha$ in $c_+$, a vector $x\in c$,
satisfying $\|Tx_\alpha-Tx\| \to 0$, need not to be an upper bound 
of $(x_\alpha)_\alpha$. Indeed, let $f$ be a positive functional on $c$ that 
assigns to any element of $c$ its limit. 
Define $T:c\to c$ by $Tx:=f(x)\cdot\mathbb{1}$, where $\mathbb{1}$ 
is the real sequence identically equal to $1$.
Then $T$ is a $KB$ operator by Proposition~\ref{prop1}. Take 
$x_n:=(\underbrace{1,1,...,1}_n,0,0,0,...)$ in $c_+$.
Then $x_n\uparrow$ and $\|x_n\|=1$ for all $n\in\mathbb{N}$.
Yet, for every $x\in c$ satisfying $x_n\uparrow \le x$, we have $\|Tx\|\ge 1$ 
and hence $\|Tx_n-Tx\|=\|Tx\| \ge 1$ for each $n$.

\begin{prop}\label{prop2}%%%OK
Every r-compact operator $T$ in a Banach lattice $E$ is regularly quasi $KB$.
\end{prop}

\begin{proof}
WLOG, we restrict ourselves to the case when $T$ 
is a positive compact operator.
Let a net $(x_\alpha)_\alpha$ in $E$ be norm bounded
and $0\le x_\alpha\uparrow$. 
Since $T$ is compact, the increasing net $(Tx_\alpha)_\alpha$ has a subnet 
$(Tx_{\alpha_\beta})_\beta$ with $\left\|Tx_{\alpha_\beta}-y\right\|\to 0$ 
for some $y \in E$. In particular, $(Tx_\alpha)_\alpha$ has a $\|.\|$-Cauchy subnet, 
and hence $(Tx_\alpha)_\alpha$ is also $\|.\|$-Cauchy, showing that $T$ is quasi $KB$, 
and since $T\ge 0$, $T$ is \text{\rm r-qKB}.
\end{proof}

\subsection{}
The next example shows that the inclusion 
$$
   \text{\rm r-L}_{KB}(E)\subseteq\text{\rm r-L}_{qKB}(E)
$$
can be proper.

\begin{exam}\label{Example1 c_0}%%%OK
{\em
Denote elements of $c_0$ by $x =\sum_{n=1}^\infty a_n \cdot e_n$,
where $e_n$ are orts placed at $n$ and $(a_n)_n$ is a real sequence converging to zero.
Define an operator $T:c_0\to c_0$ by
$$
  T\left( \sum_{n=1}^\infty a_n e_n\right)  = \sum_{n=1}^\infty \frac{a_n}{n} e_n. 
$$
Then $T$ is positive and compact, and hence a quasi $KB$ operator
by Proposition \ref{prop2}. Thus $T\in{\text{\rm r-}}{\text{\cal L}}_{qKB}(E)$.
Take an increasing norm bounded sequence $(x_n)_n$ in $c_0$ by
$x_n = \sum\limits_{k=1}^n e_k$. The sequence 
$(Tx_n)_n =\bigl(\sum\limits_{k=1}^n\frac{1}{k} e_k\bigl)_n$ is norm-Cauchy in $c_0$, 
however there is no element $\sum\limits_{k=1}^\infty a_k e_k$ in $c_0$ satisfying
$$
   T\left(\sum\limits_{k=1}^\infty a_k e_k \right)=
   \sum\limits_{k=1}^\infty \frac{1}{k}e_k=
   \|\cdot\|_{c_0}\text{-}\lim\limits_{n\to\infty}Tx_n.
$$
Therefore $T$ is not a $KB$-operator.}
\end{exam}

\subsection{}
Our next result develops the idea of Example~\ref{Example1 c_0}.

\begin{theorem}\label{prop3}%%%OK
In a Banach lattice $E$ with order continuous norm the following are equivalent$:$ 
\begin{enumerate}[$1)$]
\item \ $E$ is a $KB$-space$;$
\item \ every positive compact operator in $E$ is $KB$.
\end{enumerate}
\end{theorem}

\begin{proof}
The implication $1) \Rightarrow 2)$  is obvious, as any continuous operator 
on a $KB$-space is a $KB$-operator.

$2) \Rightarrow 1)$ \ 
Assume, in contrary, that $E$ be not a $KB$-space. 
Then there is a homeomorphic lattice embedding 
$c_0 \stackrel{i}{\longrightarrow}E$, say 
$i((a_k)_{k=1}^\infty)=\sum\limits_{k=1}^\infty a_k \cdot e_k$, 
where $(a_k)_{k=1}^\infty \in c_0$. WLOG, assume that  
the sequence $(e_k)_k$ in $E_+$ is pairwise disjoint and normalized. 
Let $f_k$ be norm-one functionals on $E$ such that $f_k(e_k)=1$ for all $k\in\mathbb{N}$.
Then $|f_k|$ are norm-one positive functionals on $E$ with 
$|f_k|(e_k)=1$ for all $k\in\mathbb{N}$.
Indeed $\||f_k|\|=\|f_k\|=1$ and 
$$
   1=|f_k(e_k)|\le|f_k|(e_k)=\sup\{|f_k(x)|: |x|\le e_k\}\le
$$
$$
   \sup\{|f_k(x)|:\|x\|\le 1\}=\|f_k\|=1
$$
for all $k\in\mathbb{N}$. Let $P_k$ be the band projection onto the principle 
band $B_{e_k}$ generated by $e_k$. Define a positive operator $T:E\to E$ by
$$
   Tx=\sum_{k=1}^\infty\frac{|f_k|(P_k(x))}{k}e_k . 
   \eqno(7)
$$
Since $||f_k|(P_k(x))|\le\||f_k|\|\cdot\|P_k(x)\|\le\|x\|$ for all $k\in\mathbb{N}$,
the formula (7) implies
$$
   T(B_E)\subseteq i([-u,u]),  
   \eqno(8)
$$
where $u:=\bigl(\frac{1}{k}\bigl)_k\in c_0$.
As order intervals of $c_0$ are compact,
the set $i([-u,u])$ is compact in $E$.
It follows from (8) that $T$ is a compact operator. 
By the assumption $2)$, $T$ is a $KB$-operator. 
Consider an increasing norm bounded sequence $y_n=\sum\limits_{k=1}^n e_k$ in $E$.
Then there is some $y\in E$ such that $\|Ty_n-Ty\|\to 0$. Since 
$$
   Ty_n=\sum_{k=1}^n\frac{|f_k|(P_k(y_n))}{k}e_k=
   \sum_{k=1}^n\frac{|f_k|(e_k)}{k}e_k=
   \sum_{k=1}^n\frac{1}{k}e_k
$$
for all $n\in\mathbb{N}$, $Ty=\sum_{k=1}^\infty\frac{1}{k}e_k$. 
Let $w_n=\sum_{k=1}^n P_k(y)$. Then $w_n\convo y$.
Since the norm in $E$ is order continuous, $\|w_n-y\|\to 0$.
Then the sequence $(w_n)_n$ is norm-Cauchy.
It follows from $\|Ty_n-Ty\|\to 0$, $Ty_n=\sum_{k=1}^n\frac{1}{k}e_k$, and
$Ty=\sum_{k=1}^\infty\frac{1}{k}e_k$ that $|f_n|(P_n(y_n))=|f_n|(P_n(y))$.
Then
$$
   1=|f_n|(e_n)=|f_n|(P_n(y_n))=|f_n|(P_n(y))\le\|P_n(y)\|,
$$
violating $\|P_n(y)\|=\|w_{n}-w_{n-1}\|\to 0$.
The obtained contradiction completes the proof.
\end{proof}

\begin{cor}\label{cor1}%%%OK
Let $E$ be a Banach lattice with order continuous norm.
The following conditions are equivalent. 
\begin{enumerate}[$1)$]
\item \ 
$E$ is a $KB$-space.
\item \
$\text{\rm L}(E)=\text{\rm L}_{KB}(E)$.
\item \ 
$\text{\rm L}_{qKB}(E)=\text{\rm L}_{KB}(E)$.
\item \ 
$\text{\rm r-L}_{qKB}(E)=\text{\rm r-L}_{KB}(E)$.
\item \ 
$(\text{\rm L}_{qKB})_+(E)=(\text{\rm L}_{KB})_+(E)$.
\item \ 
Every positive compact operator in $E$ is $KB$.
\end{enumerate}
\end{cor}

\begin{proof}
$1) \Rightarrow 2) \Rightarrow 3)$ \
is obvious.\\
$3) \Rightarrow 4)$ \ 
Let $T\in\text{\rm r-L}_{qKB}(E)$, say $T=T_1-T_2$ with 
$T_1,T_2\in(\text{\rm L}_{qKB})_+(E)$. Then
$T_1,T_2\in(\text{\rm L}_{KB})_+(E)$, and hence
$T\in\text{\rm r-L}_{KB}(E)$.\\
$4) \Rightarrow 5)$ \ 
is obvious.\\
$5) \Rightarrow 6)$ \
Let $T$ be a positive compact operator in $E$.
Let $(B_E)_+\ni x_\alpha\uparrow$. By compactness of $T$,
the net $(Tx_\alpha)_\alpha$ contains a norm convergent 
subnet, and, since $Tx_\alpha\uparrow$ then the net $(Tx_\alpha)_\alpha$
is norm convergent. Thus $T\in (\text{\rm L}_{qKB})_+(E)$, and hence 
$T\in (\text{\rm L}_{KB})(E)$.\\
$6) \Rightarrow 1)$ follows from Theorem \ref{prop3}.
\end{proof}
\noindent
In view of Example \ref{quasi KB yet not KB} 
the assumption that the norm in $E$ is order continuous
is essential in Corollary \ref{cor1}. The author 
do not know whether or not the order continuity in Corollary \ref{cor1}
can be replaced by $\sigma$-Dedekind completeness, 
or by Dedekind completeness.

{
}
\end{document}